\documentclass[11pt,draft]{amsart}

\usepackage{a4,amssymb,amsmath,cite}
\usepackage{graphicx,gastex}

\usepackage{tikz}
\usetikzlibrary{positioning}
\usepackage{amsmath}
\usetikzlibrary{arrows.meta}
\usepackage{verbatim}

\DeclareMathOperator{\con}{con}

\DeclareMathOperator{\alphabet}{alph}

\DeclareMathOperator{\Con}{Con}

\DeclareMathOperator{\SCon}{SCon}

\DeclareMathOperator{\Part}{Part}

\DeclareMathOperator{\End}{End}

\DeclareMathOperator{\Aut}{Aut}

\DeclareMathOperator{\I}{I}

\DeclareMathOperator{\Sub}{Sub}

\DeclareMathOperator{\Stab}{Stab}

\DeclareMathOperator{\Orb}{Orb}

\DeclareMathOperator{\Code}{Code}

\newtheorem{theorem}{Theorem}[section]

\newtheorem{proposition}[theorem]{Proposition}

\newtheorem{lemma}[theorem]{Lemma}

\newtheorem{corollary}[theorem]{Corollary}

\begin{document}

\title[Modular elements of the lattices of varieties]{Modular elements of the lattices of varieties of semigroups and epigroups. I}

\author{Vyacheslav Yu. Shaprynski\v{\i}}
\thanks{V. Yu. Shaprynski\v{\i}: Institute of Natural Sciences and Mathematics, Ural Federal University, Ekaterinburg 620000, Russia. \texttt{vshapr@yandex.ru}
\\The research of the first author was supported by the Ministry of Science and Higher Education of the Russian Federation (project FEUZ-2023-0022).}

\author{Dmitry V. Skokov}
\thanks{D. V. Skokov: Department of Mathematics, Bar-Ilan University, Ramat Gan 5290002, Israel. \texttt{dmitry.skokov@gmail.com} \\The research of the second author was supported by the scholarship of the Center for Absorption in Science, the Ministry for Absorption of Aliyah, the State of Israel.}

\date{}

\keywords{Semigroup, epigroup, variety, lattice of subvarieties, special element, modular element}

\subjclass{Primary 20M07, secondary 08B15}

\begin{abstract}
This paper is the first part of a study devoted to description of modular elements in the lattices of semigroup and epigroup varieties. We provide strengthened necessary and sufficient conditions under which a semigroup or epigroup variety constitutes a modular element in its respective lattice. These results refine previously known criteria and lay the groundwork for a complete classification, to be presented in the second part of the study.
\end{abstract}

\maketitle

\section{Introduction}
\label{introduction}

A \emph{variety} is a class of universal algebras of a fixed type that is closed under homomorphic images, subalgebras, and arbitrary direct products. This work continues the long-standing and active study of varieties of semigroups. It is known that the collection of all semigroup varieties forms a lattice, denoted by~$\mathbb{SEM}$, with respect to class-theoretical inclusion. For over six decades, the structure and properties of the lattice $\mathbb{SEM}$ have drawn significant attention from researchers, underscoring its pivotal role in semigroup theory. The foundational results from the early stages of this research are summarized in the survey~\cite{Evans-71}, while a comprehensive overview of more recent developments can be found in~\cite{Shevrin-Vernikov-Volkov-09}.

The structure of $\mathbb {SEM}$ is remarkably intricate. One notable feature is that $\mathbb {SEM}$ contains an anti-isomorphic copy of the partition lattice over a countably infinite set~\cite{Burris-Nelson-71, Jezek-76}. As a consequence, $\mathbb {SEM}$ does not satisfy any non-trivial lattice identities. The study of identities in subvariety lattices of semigroup varieties has been the subject of extensive research, producing numerous profound and impactful results (see~\cite[Section 11]{Shevrin-Vernikov-Volkov-09}). A natural progression from these studies is the exploration of varieties that exhibit what could be described as ‘nice lattice behavior’ within their local neighborhoods. In particular, our focus turns to the study of \emph{special elements} within the lattice $\mathbb {SEM}$, which possess properties connected to lattice identities. While certain classes of special elements have been fully described, others remain only partially understood. Vernikov's 2015 survey~\cite{Vernikov-15} provides a comprehensive summary of the results achieved up to that time.

Here our main goal is to get a complete description of one of the types of special elements in $\mathbb {SEM}$ -- the modular elements. An element $x$ of a lattice $\left\langle {L, \wedge, \vee} \right\rangle$ is called \emph{modular} if $$(\forall y, z \in L) \,\,\, y \le z \longrightarrow (x \vee y) \wedge z = (x \wedge z) \vee y.$$ An equivalent characterization of modular elements was established by Mikhail Volkov in \cite[Proposition 1.1]{Volkov-05}: an element $x$ of a lattice $\langle L, \wedge, \vee \rangle$ is modular if and only if it is not the central element of the non-modular lattice $N_5$ (see Fig.~\ref{Pent}).

\begin{figure}[tbh]
\centering
\begin{tikzpicture}[scale=1.7,
  dot/.style={circle, fill=black, inner sep=1.6pt}
]
  \coordinate (A) at (0,1);
  \coordinate (B) at (-0.7,0.5);
  \coordinate (C) at (-0.7,0);
  \coordinate (D) at (0,-0.5);
  \coordinate (E) at (0.7,0.25); 

  \draw[line width=0.9pt] (A)--(B)--(C)--(D)--(E)--(A);

  \foreach \p in {A,B,C,D,E}{\node[dot] at (\p) {};}

  \draw[-{Latex[length=2.4mm]}, line width=0.9pt]
    (1.25,0.45) .. controls (1.05,0.3) and (1.0,0.28) .. (E);
  \node[anchor=west] at (1.3,0.45) {modular elements cannot be here};
\end{tikzpicture}
\caption{The non-modular lattice $N_5$}
\label{Pent}
\end{figure}
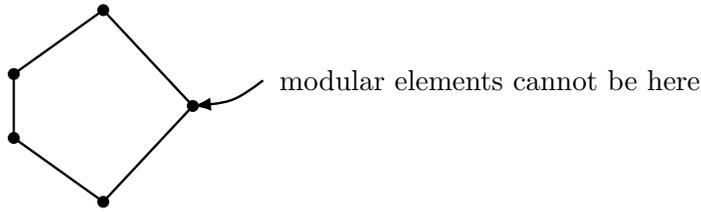

 Among the various types of special elements, modular elements have garnered significant interest. These elements were among the first special elements in $\mathbb {SEM}$ to attract the attention of researchers. Despite this long-standing interest, a complete characterization of modular elements in $\mathbb {SEM}$ has remained an open problem until recently. It is worth noting that the problem of obtaining a complete description of modular elements in the lattice $\mathbb {SEM}$ was explicitly formulated in several papers (see, for instance,~\cite{Vernikov-07}, including the fundamental survey~\cite{Shevrin-Vernikov-Volkov-09}).

The results obtained in the study of the lattice $\mathbb {SEM}$ are often useful for understanding the structure of lattices of varieties of semigroups equipped by an additional unary operation. Such algebras are said to be \emph{unary semigroups}. One natural type of unary semigroups is epigroups. 

A semigroup $S$ is called an \emph{epigroup} if, for any element $x$ of $S$, there is a positive integer $n$ such that $x^n$ is a group element (that is, lies in some subgroup of~$S$). Extensive information about epigroups may be found in the fundamental work~\cite{Shevrin-94} by L.\,N.\,Shevrin and the survey~\cite{Shevrin-05} by the same author. The class of epigroups is very wide. In particular, it includes all periodic semigroups (because some power of each element in such a semigroup lies in some its finite cyclic subgroup) and all completely regular semigroups (in which all elements are group elements). 

The unary operation on an epigroup is defined in the following way. For an element $x$ of a given epigroup, let $e_x$ be the unit element of the maximal subgroup $G$ that contains some power of $x$. It is known (see~\cite{Shevrin-94}, for instance) that $xe_x = e_xx$ and this element lies in $G$. We denote by $\overline{x}$ the inverse element to $xe_x$ in $G$. This element is called the \emph{pseudoinverse} of $x$. The mapping $x\longmapsto\overline x$ defines a unary operation on an epigroup. Throughout this paper, we consider epigroups as algebras with the operations of multiplication and pseudoinversion. This naturally leads to the concept of varieties of epigroups as algebras with these two operations. The idea of examining epigroups within the framework of the theory of varieties was promoted by L.\,N.\,Shevrin in~\cite{Shevrin-94} (see also~\cite{Shevrin-05}). 

The class of all epigroups is not a variety because it is not closed under Cartesian product. This means, for instance, that the lattice of all epigroup varieties does not contain the largest element. This lattice will be denoted by $\mathbb{EPI}$. An examination of the lattice $\mathbb{EPI}$ was initiated in~\cite{Shevrin-94}. An
overview of the first results obtained in this area can be found in~\cite[Section 2]{Shevrin-Vernikov-Volkov-09}.

It is well known and can be easily checked that in every periodic epigroup variety the
operation of pseudoinversion can be expressed in terms of multiplication (see~\cite{Shevrin-94}
or~\cite{Shevrin-05}, for instance). This means that periodic varieties of epigroups can be identified with periodic varieties of semigroups. Thus, the lattices $\mathbb {SEM}$ and $\mathbb{EPI}$ have a
big common sublattice, namely, the lattice of all periodic semigroup varieties.
Results of the mentioned above article~\cite{Jezek-76} immediately imply that even the lattice
of all periodic semigroup varieties contains an anti-isomorphic copy of the partition lattice over a countably infinite set. This means, in particular, that the lattice $\mathbb{EPI}$ also contains the dual to this partition lattice as a sublattice. Therefore,  $\mathbb{EPI}$, as well as $\mathbb {SEM}$, does not satisfy any non-trivial lattice identity.

Recently, a number of articles devoted to the study of special elements in the lattice $\mathbb{EPI}$ have appeared~\cite{Skokov-15, Skokov-16, Shaprynskii-Skokov-Vernikov-16, Shaprynskii-Skokov-Vernikov-19}. The present article continues these investigations. Here our main goal is to get a complete description of modular elements in the lattices $\mathbb{SEM}$ and $\mathbb{EPI}$. For convenience, we refer to a semigroup variety [an epigroup variety] as \emph{modular} if it is a modular element of the lattice $\mathbb{SEM}$ [respectively, of the lattice $\mathbb{EPI}$].

According to the famous formulation of Birkhoff’s Variety Theorem (also known as the HSP Theorem, originally published in~\cite{Birkhoff-35}; see also~\cite{Gratzer-79} and~\cite{Burris-Sankappanavar-81}), a variety of algebras can be defined as the class of all algebraic structures of a fixed signature that satisfy a given set of identities. In this work, we consider identities involving a standard associative binary operation for semigroup varieties, and identities involving both this binary operation and an additional unary operation for epigroup varieties.

Elements of the free unary semigroup of countably infinite rank are called \emph{words}. Words, unlike letters (i.e., elements of the alphabet), are written in bold. A word that does not contain the unary operation is called a \emph{semigroup word}, and it is natural to regard such words as elements of the free semigroup. We use the symbol $\approx$ to denote identities between words, while the symbol  $=$ denotes equality in the free (unary) semigroup or in other standard contexts.

Suppose a letter $x$ does not occur in a word $\mathbf w$. A semigroup $S$ satisfies the identity system $\mathbf wx\approx x\mathbf{w\approx w}$ if and only if $S$ contains a zero element~0 and all values of $\mathbf w$ in $S$ are equal to~0. We adopt the usual convention of writing $\mathbf w\approx 0$ as a short form of such a system and referring to the expression $\mathbf w\approx 0$ as to a single identity. Identities of the form $\mathbf w\approx 0$ and varieties given by such identities are called 0-\emph{reduced}. If $\mathbf u$ is a semigroup word then its length is denoted by $\ell(\mathbf u)$. The \emph{alphabet} of a word $\mathbf w$, i.e., the set of all letters occurring in $\mathbf w$, is denoted by $\alphabet(\mathbf w)$.

We denote by $S_n$ the symmetric group on the set $\{1,2,\dots,n\}$. More generally, $S(X)$ is the symmetric group on a set $X$. We fix notation for the following subgroups of $S_n$: $T_{ij}$ is the group generated by a transposition $(ij)$; $A_n$ is the alternating subgroup of $S_n$; $I_{ij,st}$ is the subgroup of $S_4$ generated by the disjoint transpositions $(ij)$, $(st)$, and their product $(ij)(st)$. The subgroup lattice of a group $G$ is denoted by $\Sub(G)$. For any semigroup (epigroup) variety $\mathbf V$ and any word $\mathbf u$ put $$\Stab_{\mathbf V} (\mathbf u) = \{\sigma\in S(\alphabet(\mathbf u))\mid\mathbf u \approx\sigma(\mathbf u) \,\,\text{holds in}\,\,\mathbf V\}.$$ Here $\sigma(\mathbf u)$ means that the permutation $\sigma$ acts on words by renaming letters. It is clear that $\Stab_{\mathbf V} (\mathbf u)$ is a subgroup of $S(\con(\mathbf u))$. We call this subgroup the $\mathbf V$-\emph{stabilizer} of the word $\mathbf u$.

An identity $\mathbf{u\approx v}$ is called \emph{substitutive} if $\alphabet(\mathbf u)=\alphabet(\mathbf v)$ and the word $\mathbf v$ is obtained from $\mathbf u$ by renaming letters.

Let $F$ be the free (non-unary) semigroup of countably infinite rank, $F^1$ the free monoid. We denote by $\Aut(F)$ and $\End(F)$ the group of automorphisms and the monoid of endomorphisms on $F$ respectively. Let $\mathbf u,\mathbf v\in F$. We write $\mathbf{u\le v}$ if $\mathbf{v=a}\xi(\mathbf u)\mathbf b$ for some $\xi\in\End(F)$ and some $\mathbf a,\mathbf b\in F^1$. We will say that words $\mathbf u$ and $\mathbf v$ are \emph{equivalent} and write $\mathbf u\sim\mathbf v$ if $\mathbf{u\le v}$ and $\mathbf{v\le u}$. It is clear that words $\mathbf u$ and $\mathbf v$ are equivalent if and only if $\mathbf v=\varphi(\mathbf u)$ for some $\varphi\in\Aut(F)$. If $\mathbf{u\le v}$ and $\mathbf{u\not\sim v}$ then we write $\mathbf{u<v}$. We say that words $\mathbf u$ and $\mathbf v$ are \emph{incomparable} and write $\mathbf u\parallel\mathbf v$ if $\mathbf{u\not\le v}$ and $\mathbf{v\not\le u}$.

Let \( \mathbf{T} \) be the trivial variety of semigroups or epigroups. A semigroup variety is called \emph{proper} if it is distinct from the variety of all semigroups. Recall that there is no variety of all epigroups; therefore, there is no real reason to use the term \emph{proper} for epigroup varieties. However, the term may occasionally appear in order to unify the formulations concerning modular elements in the lattices \( \mathbb{SEM} \) and \( \mathbb{EPI} \). It is evident that the variety of all semigroups, being the greatest element in the lattice \( \mathbb{SEM} \), is a modular element of this lattice.

Recall that a semigroup $S$ with a zero element $0$ is said to be a \emph{nilsemigroup} if for
every $s \in S$ there exists a positive integer $n$ such that $s^n = 0$. A semigroup [epigroup]
variety consisting of nilsemigroups is called a \emph{nil-variety}.

By Proposition~\ref{mod nec}, the problem of describing modular elements of the lattices $\mathbb{SEM}$ and $\mathbb{EPI}$ reduces to describing modular nil-varieties. Indeed, every proper modular element in $\mathbb{SEM}$ [respectively, $\mathbb{EPI}$] is the join of either the trivial variety $\mathbf{T}$ or the semilattice variety $\mathbf{SL}$ with a nil-variety. Therefore, the characterization of modular nil-varieties yields a complete description of all modular elements in these lattices.

In order to formulate the characterization, we consider the following conditions on a variety $\mathbf{V}$ of semigroups [respectively, epigroups]:
\begin{itemize}
\item[a)] for each non-substitutive identity $\mathbf u\approx\mathbf v$ which holds in $\mathbf V$, the identities $\mathbf u\approx0$ and $\mathbf v\approx 0$ also hold in $\mathbf V$; 
\item[b)] for each word $\mathbf u$, if $\mathbf V$ does not satisfy the identity $\mathbf u \approx 0$ then $\Stab_{\mathbf V}(\mathbf u)$ is a modular element of the lattice $\Sub(S(\alphabet(\mathbf u)))$;
\item[c)\text{[c')]}] there are no incomparable [non-equivalent] words $\mathbf u$ and $\mathbf v$ such that $\alphabet(\mathbf u) = \alphabet(\mathbf v)$ and one of the following holds
\begin{itemize}
\item $\Stab_{\mathbf V} (\mathbf u), \Stab_{\mathbf V} (\mathbf v) \in \{ T_{12},  T_{23},  T_{13}\}$;
\item $\Stab_{\mathbf V} (\mathbf u) \in \{ T_{12},  T_{23},  T_{13}\}, \Stab_{\mathbf V} (\mathbf v) =  A_3$;
\item $\Stab_{\mathbf V} (\mathbf u), \Stab_{\mathbf V} (\mathbf v) \in \{ I_{12,34},  I_{13,24},  I_{14,23}\}$;
\item $\Stab_{\mathbf V} (\mathbf u) \in \{ I_{12,34},  I_{13,24},  I_{14,23}\}, \Stab_{\mathbf V} (\mathbf v) =  A_4$.
\end{itemize}
\end{itemize}

We can now state our main results.

\begin{theorem}
\label{necessity}
Let $\mathbf{V}$ be a proper variety of semigroups \textup[respectively, epigroups\textup] which is a modular element of the lattice $\mathbb{SEM}$ \textup[respectively, $\mathbb{EPI}$\textup]. Then
\(
\mathbf{V} = \mathbf{M} \vee \mathbf{N},
\)
where $\mathbf{M} \in \{\mathbf{T}, \mathbf{SL}\}$ and $\mathbf{N}$ is a nil-variety satisfying conditions \textup{(a)}, \textup{(b)}, and \textup{(c)} above.
\end{theorem}

\begin{theorem}
\label{sufficiency}
Let $\mathbf{M} \in \{\mathbf{T}, \mathbf{SL}\}$ and let $\mathbf{N}$ be a nil-variety of semigroups \textup[respectively, epigroups\textup] satisfying conditions \textup{(a)}, \textup{(b)}, and \textup{(c')} above. Then
\(
\mathbf{V} = \mathbf{M} \vee \mathbf{N}
\)
is a modular element of the lattice $\mathbb{SEM}$ \textup[respectively, $\mathbb{EPI}$\textup].
\end{theorem}

This paper is the first of two parts in our investigation.
Here, Theorems~\ref{necessity} and~\ref{sufficiency} provide a necessary and a sufficient conditions for a variety to be modular.
Although stated separately, these conditions are in fact equivalent --- a fact that will be established in the second part of this work. In particular, this will imply that modular varieties of semigroups and epigroups are the same (except the variety of all semigroups). The complete characterization of modular elements will also be given in the second part, where the reader will find a full description of modular elements in the lattices under consideration.

Recently, the description of the modular elements of the lattice of varieties of monoids was obtained by Sergey Gusev~\cite{Gusev-26}.  
There is a significant difference between the semigroup and the monoid cases.  
Namely, in~\cite{Gusev-26} it is proved that, in the lattice of varieties of monoids, the modular elements form a sublattice.  
However, this property does not hold for the lattice of semigroup (or epigroup) varieties.  Let
\(
\mathbf{V}_1\) be a varity generated by identities \( \{x^2yz \approx x^2zy,\; x_1x_2x_3x_4x_5 \approx 0\}\), 
 and  
\(\mathbf{V}_2\) be a varity generated by identities \( \{xyz^2 \approx yxz^2,\; x_1x_2x_3x_4x_5 \approx 0\}.
\)
Using Theorems~\ref{necessity} and~\ref{sufficiency}, one can verify that the varieties $\mathbf V_1$ and $\mathbf V_2$ are modular while the variety $\mathbf{V}_1\wedge\mathbf{V}_2$ is not.

\section{Preliminaries}
\label{Preliminaries}

\subsection{Modular elements in an arbitrary lattice}
\label{modular an arbitrary lattice}

We begin with a general observation on how modularity is preserved under surjective lattice homomorphisms.

\medskip

\begin{lemma}
\label{surjective homomorphism}
Let $L_1$ and $L_2$ be arbitrary lattices, and let $\chi \colon L_1 \to L_2$ be a surjective lattice homomorphism. If $a \in L_1$ is a modular element in $L_1$, then $\chi(a)$ is a modular element in $L_2$. 
\end{lemma}

\begin{proof}
Let $x, y \in L_2$ be such that $x \le y$. We aim to prove that
\[
(\chi(a) \vee x) \wedge y = (\chi(a) \wedge y) \vee x.
\]
Since $\chi$ is surjective, there exist elements $u, v \in L_1$ such that $\chi(u) = x$ and $\chi(v) = y$. Clearly,
\[
\chi(u \vee v) = \chi(u) \vee \chi(v) = x \vee y = y.
\]
Using the fact that $a$ is modular in $L_1$ and that $u \le u \vee v$, we obtain
\begin{align*}
(\chi(a) \vee x) \wedge y 
&= (\chi(a) \vee \chi(u)) \wedge \chi(u \vee v) && \text{(by choice of $u, v$)} \\
&= \chi((a \vee u) \wedge (u \vee v)) && \text{($\chi$ is a homomorphism)} \\
&= \chi((a \wedge (u \vee v)) \vee u) && \text{($a$ is modular)} \\
&= (\chi(a) \wedge \chi(u \vee v)) \vee \chi(u) && \text{($\chi$ is a homomorphism)} \\
&= (\chi(a) \wedge y) \vee x.
\end{align*}
This completes the proof.
\end{proof}

\subsection{Modular elements in subgroup lattices of finite symmetric groups}
\label{modular in subgroup}

Let $T$ be the trivial group, $T_{ij}$ be the group generated by a transposition $(ij)$, $C_{ijk}$ and $C_{ijk\ell}$ be the groups generated by cycles $(ijk)$ and $(ijk\ell)$ respectively, $P_{ij,k\ell}$ be the group generated by the product of disjoint transpositions $(ij)$ and $(k\ell)$, $A_n$ be the alternating subgroup of $S_n$ and $V_4$ be the Klein four-group. The subgroup lattice of a group $G$ is denoted by $\Sub(G)$. 
Let us recall the structure of the lattices $\Sub(S_3)$ and $\Sub(S_4)$. It is generally known and easy to check that the first of these two lattices has the form shown in Fig.~\ref{Sub(S_3)}. Direct routine calculations allow to verify that the lattice $\Sub(S_4)$ is as shown in Fig.~\ref{Sub(S_4)}. 

\begin{figure}[tbh]
\centering
\begin{tikzpicture}[scale=1, every node/.style={circle, fill=black, inner sep=1.5pt}]
  \node (T12) at (0,0) {};
  \node (T13) at (2,0) {};
  \node (T)   at (3,-2) {};
  \node (S3)  at (3,2) {};
  \node (T23) at (4,0) {};
  \node (C123) at (6,0) {};

  \node[draw=none, fill=none, left=2pt of T12]  {$T_{12}$};
  \node[draw=none, fill=none, left=2pt of T13]  {$T_{13}$};
  \node[draw=none, fill=none, right=2pt of T23] {$T_{23}$};
  \node[draw=none, fill=none, right=2pt of C123]{$C_{123}$};
  \node[draw=none, fill=none, above=2pt of S3]  {$S_3$};
  \node[draw=none, fill=none, below=2pt of T]   {$T$};

  \draw (T) -- (T12) -- (S3) -- (T13) -- (T);
  \draw (T) -- (T23) -- (S3) -- (C123);
  \draw (T) -- (C123);
\end{tikzpicture}
\caption{The lattice $\Sub(S_3)$}
\label{Sub(S_3)}
\end{figure}
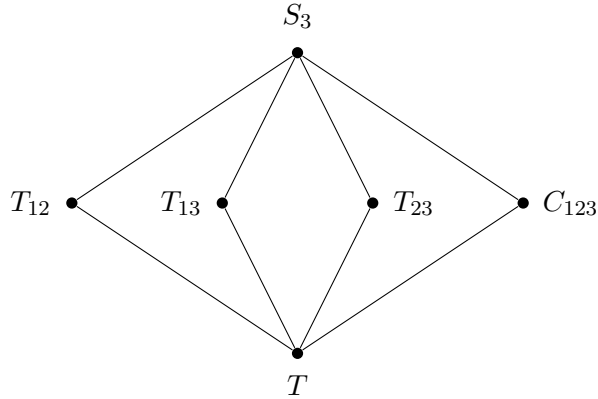

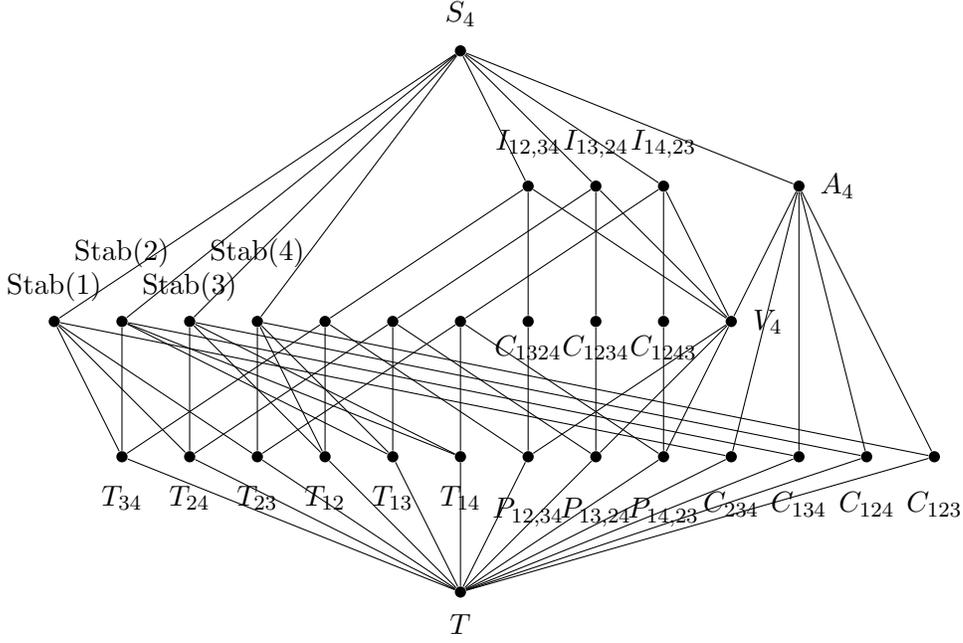
\begin{figure}[tbh]
\centering
\begin{tikzpicture}[scale=0.9, every node/.style={circle, fill=black, inner sep=1.5pt}]
 \node (T) at (6,0) {};
 \node (T34) at (1,2) {};
 \node (T24) at (2,2) {};
 \node (T23) at (3,2) {};
 \node (T12) at (4,2) {};
 \node (T13) at (5,2) {};
 \node (T14) at (6,2) {};
 \node (P1234) at (7,2) {};
 \node (P1324) at (8,2) {};
 \node (P1423) at (9,2) {};
 \node (C234) at (10,2) {};
 \node (C134) at (11,2) {};
 \node (C124) at (12,2) {};
 \node (C123) at (13,2) {};

 \node (Stab1) at (0,4) {}; 
\node (Stab2) at (1,4) {};
 \node (Stab3) at (2,4) {};
 \node (Stab4) at (3,4) {};
 \node (Point1234) at (4,4) {};
 \node (Point1324) at (5,4) {};
 \node (Point1423) at (6,4) {};
 \node (C1234) at (7,4) {};
 \node (C1324) at (8,4) {};
 \node (C1423) at (9,4) {};
 \node (V4) at (10,4) {};

 \node (I1234) at (7,6) {};
 \node (I1324) at (8,6) {};
 \node (I1423) at (9,6) {};

 \node (A4) at (11,6) {};

 \node (S4) at (6,8) {};

 \node[draw=none, fill=none, below=-8pt of C1234]  {$C_{1324}$};
 \node[draw=none, fill=none, below=-8pt of C1324]  {$C_{1234}$};
 \node[draw=none, fill=none, below=-8pt of C1423]  {$C_{1243}$};\
 \node[draw=none, fill=none, right=2pt of V4]  {$V_{4}$};
 \node[draw=none, fill=none, below=2pt of C134]  {$C_{134}$};
 \node[draw=none, fill=none, below=2pt of C124]  {$C_{124}$};
 \node[draw=none, fill=none, below=2pt of C123]  {$C_{123}$};
 \node[draw=none, fill=none, below=2pt of C234]  {$C_{234}$};
 \node[draw=none, fill=none, below=2pt of P1423]  {$P_{14,23}$};
 \node[draw=none, fill=none, below=2pt of P1324]  {$P_{13,24}$};
 \node[draw=none, fill=none, below=2pt of P1234]  {$P_{12,34}$};
 \node[draw=none, fill=none, below=2pt of T12]  {$T_{12}$};
  \node[draw=none, fill=none, below=2pt of T13]  {$T_{13}$};
  \node[draw=none, fill=none, below=2pt of T23] {$T_{23}$};
  \node[draw=none, fill=none, below=2pt of T34]{$T_{34}$};
  \node[draw=none, fill=none, below=2pt of T24]  {$T_{24}$};
  \node[draw=none, fill=none, below=2pt of T14]  {$T_{14}$};
  \node[draw=none, fill=none, below=2pt of T]   {$T$};

 \node[draw=none, fill=none, above=-2pt of I1234]   {$I_{12,34}$};
 \node[draw=none, fill=none, above=-2pt of I1324]   {$I_{13,24}$};
 \node[draw=none, fill=none, above=-2pt of I1423]   {$I_{14,23}$};
 \node[draw=none, fill=none, right=2pt of A4]   {$A_4$};

\node[draw=none, fill=none, above=-10pt of Stab1]   {$\Stab (1)$};
 \node[draw=none, fill=none, above=3pt of Stab2]   {$\Stab (2)$};
 \node[draw=none, fill=none, above=-10pt of Stab3]   {$\Stab (3)$};
 \node[draw=none, fill=none, above=3pt of Stab4]   {$\Stab (4)$};

 \node[draw=none, fill=none, above=2pt of S4]   {$S_4$};

  \draw (T) -- (T34) -- (Stab1) -- (S4);
 \draw (T34) -- (Stab2);
  \draw (T) -- (T24) -- (Stab1);
  \draw (T24) -- (Stab3);
  \draw (T) -- (T23) -- (Stab1);
  \draw (T23) -- (Stab4);
 \draw (T) -- (T12) -- (Stab3)  -- (S4);
 \draw (T12) -- (Stab4) -- (S4);
  \draw (T) -- (T13) -- (Stab2) -- (S4);
  \draw (T13) -- (Stab4);
  \draw (T) -- (T14) -- (Stab2);
  \draw (T14) -- (Stab3);
 \draw (T) -- (P1234) -- (V4) -- (A4) -- (S4);
 \draw (T) -- (P1324)-- (V4);
 \draw (T) -- (P1423)-- (V4);
 \draw  (V4) -- (I1324);
 \draw  (V4) -- (I1234);
 \draw  (V4) -- (I1423);
 \draw  (P1234) -- (C1234) -- (I1234) -- (S4);
 \draw  (P1324) -- (C1324) -- (I1324) -- (S4);
 \draw  (P1423) -- (C1423) -- (I1423) -- (S4);
 \draw (T) -- (C234) -- (Stab1);
 \draw (C234) -- (A4);
 \draw (T) -- (C134) -- (Stab2);
 \draw (C134) -- (A4);
 \draw (T) -- (C124) -- (Stab3);
 \draw (C124) -- (A4);
 \draw (T) -- (C123) -- (Stab4);
 \draw (C123) -- (A4);
\draw (T34) -- (Point1234) -- (I1234);
\draw (T12) -- (Point1234);
\draw (P1234) -- (Point1234);
\draw (T13) -- (Point1324) -- (I1324);
\draw (T24) -- (Point1324);
\draw (P1324) -- (Point1324);
\draw (T14) -- (Point1423) -- (I1423);
\draw (T23) -- (Point1423);
\draw (P1423) -- (Point1423);
\end{tikzpicture}
\caption{The lattice $\Sub(S_4)$}
\label{Sub(S_4)}
\end{figure}

Modular elements of the lattice $\Sub(S_n)$ for an arbitrary $n$ are completely determined in~\cite[Propositions~3.1,~3.7 and~3.8]{Jezek-81}. If $n\le 2$ then $S_n$ does not contain any subgroups different from $T$ and $S_n$. If $n=3$ then all subgroups are modular which follows from Fig.~\ref{Sub(S_3)}. If $n=4$ then a non-singleton proper subgroup $G$ of $S_n$ is modular if and only if $G\supseteq V_4$ (see Fig.~\ref{Sub(S_4)}). If $n\ge 5$ then a non-trivial subgroup $G$ of $S_n$ is modular if and only if  $G\supseteq A_n$ (that is, $G = S_n$ or $G = A_n$).

\subsection{Joining with the variety of all semilattices} 
\label{Joining with the variety of all semilattices}

Recall that an element $x$ of a lattice $\langle L, \vee, \wedge \rangle$ is called \emph{neutral} if
\[
(\forall y, z \in L)\quad (x \vee y) \wedge (y \vee z) \wedge (z \vee x) = (x \wedge y) \vee (y \wedge z) \vee (z \wedge x).
\]

We begin with a general result concerning joins with neutral elements in lattices.

\medskip

\begin{lemma}[\!\!{\mdseries\cite[Lemma~1.6]{Volkov-05}}]
\label{modular atom}
Let $a$ be a neutral element of a lattice $L$. If $x \in L$ is a modular element of $L$, then so is $x \vee a$. \qed
\end{lemma}

\medskip

Let \( \mathbf{SL} \) denote the variety of all semilattices, that is, the variety defined by the identities
\[
x^2 \approx x, \quad xy \approx yx.
\]
It is well known that \( \mathbf{SL} \) is a neutral element of both \( \mathbb{SEM} \)~\cite[Proposition~4.1]{Volkov-05} and \( \mathbb{EPI} \)~\cite[Theorem~1.1]{Shaprynskii-Skokov-Vernikov-16}. Combining this fact with Lemma~\ref{modular atom}, we immediately obtain the following result.

\medskip

\begin{lemma}
\label{join with SL}
A variety of semigroups \textup[respectively, epigroups\textup] $\mathbf{V}$ is a modular element of the lattice $\mathbb{SEM}$ \textup[respectively, $\mathbb{EPI}$\textup] if and only if $\mathbf{V} \vee \mathbf{SL}$ is a modular element in $\mathbb{SEM}$ \textup[respectively, $\mathbb{EPI}$\textup]. \qed
\end{lemma}

\medskip

The following proposition provides a stronger version of the necessary condition given by Lemma~\ref{join with SL}. 

\medskip

\begin{proposition}
\label{mod nec}
If a proper variety of semigroups \textup[epigroups\textup] $\mathbf V$ is a modular element of the lattice $\mathbb{SEM}$ \textup[respectively $\mathbb{EPI}$\textup] then $\mathbf{V=M\vee N}$ where $\mathbf M$ is one of the varieties $\mathbf T$ and $\mathbf{SL}$, while $\mathbf N$ is a nil-variety.\qed
\end{proposition}

\medskip

The semigroup part of Proposition~\ref{mod nec} was established --- though in a slightly weaker form and using different terminology ---  in~\cite[Proposition~1.6]{Jezek-McKenzie-93}. A deduction of Proposition~\ref{mod nec} from~\cite[Proposition~1.6]{Jezek-McKenzie-93} was given explicitly in~\cite[Proposition~2.1]{Vernikov-07}. A direct and transparent proof of the ``semigroup half'' of Proposition~\ref{mod nec} not depending on a technique from~\cite{Jezek-McKenzie-93} is provided in~\cite{Shaprynskii-12}. The epigroup part of Proposition~\ref{mod nec} corresponds to a weaker form of~\cite[Theorem~1.2]{Shaprynskii-Skokov-Vernikov-16}.

In \cite{Jezek-81}, Je\v{z}ek describes the modular elements of the lattice of all varieties (more precisely, all equational theories) of a given type. In particular, it follows from \cite[Lemma~6.3]{Jezek-81} that if a nil-variety of semigroups $\mathbf V$ is a modular element of the lattice of all groupoid varieties then $\mathbf V$ may be given by 0-reduced and substitutive identities only. This does not imply directly the same conclusion for modular nil-varieties because a modular element of $\mathbb {SEM}$ need not be a modular element of the lattice of all groupoid varieties. 

Proposition~\ref{mod nec} and Lemma~\ref{join with SL} completely reduce the study of modular varieties to the case of nil-varieties.
As a consequence of the preceding discussion, we now formulate the following proposition, which was proved in~\cite[Theorem~2.5]{Vernikov-07} for the semigroup case, and in~\cite[Theorem~1.6]{Shaprynskii-Skokov-Vernikov-16} for the epigroup case.

\medskip

\begin{proposition}
\label{SEM cmod nil-nec}
Suppose a proper variety of semigroups or epigroups \( \mathbf{V} \) is a modular element of the respective lattice. If $\mathbf V$ satisfies an identity $\mathbf u\approx\mathbf v$ but does not satisfy $\mathbf u\approx 0$ then the identity $\mathbf u\approx\mathbf v$ is substitutive.
\qed
\end{proposition}

\medskip

A periodic variety of epigroups can be defined by identities involving only multiplication (see, for instance,\cite{Shevrin-94} or\cite{Shevrin-05}). Consequently, periodic varieties of epigroups may be identified with periodic varieties of semigroups. Therefore, in what follows, we consider only semigroup identities when dealing with both epigroup and semigroup varieties.

Proposition~\ref{mod nec} immediately implies

\begin{corollary}
\label{periodic}
If a semigroup \textup[epigroup\textup] variety \( \mathbf{V} \) is a modular element of the lattice \( \mathbb{EPI} \), then \( \mathbf{V} \) is periodic. \qed
\end{corollary}

\medskip

Proposition~\ref{SEM cmod nil-nec} provides a strong necessary condition for a variety to be modular. The following proposition gives some sufficient condition for modular varieties.

\medskip

\begin{proposition}
\label{SEM cmod suf}
Every $0$-reduced variety of semigroups \textup(or epigroups\textup) is a modular element of the lattice $\mathbb{SEM}$ \textup(respectively, $\mathbb{EPI}$\textup). \qed
\end{proposition}

\medskip

For semigroup varieties, this  proposition was noted for the first time in~\cite[Corollary~3]{Vernikov-Volkov-88} and rediscovered (in different terminology) in~\cite[Proposition~1.1]{Jezek-McKenzie-93}. For epigroup varieties, it was noted in~\cite{Shaprynskii-Skokov-Vernikov-16}.

Propositions~\ref{SEM cmod nil-nec} and~\ref{SEM cmod suf} provide a necessary and a sufficient condition for a variety to be modular respectively. The gap between these conditions seems to be not very large. But the necessary condition is not a sufficient one, while the sufficient condition is not a necessary one (this follows from Proposition~\ref{SEM cmod commut} below).

\bigskip

 A natural and important subclass of the class of substitutive identities is the class of permutational identities. The identity
\begin{equation}
\label{permut id}
x_1x_2\cdots x_m\approx x_{1\pi}x_{2\pi}\cdots x_{m\pi}
\end{equation}
where $\pi\in S_m$ is denoted by $p_m[\pi]$. If the permutation $\pi$ is non-trivial then this identity is called \emph{permutational}. The number $m$ is called the \emph{length} of this identity. The strongest permutational identity is the commutative law. Modular varieties satisfying this law are completely classified by the following

\medskip

\begin{proposition}[\!\!{\mdseries\cite[Theorem~3.1]{Vernikov-07}}]
\label{SEM cmod commut}
A commutative semigroup variety $\mathbf V$ is a modular element of $\mathbb{SEM}$ if and only if $\mathbf{V=M\vee N}$ where $\mathbf M$ is one of the varieties $\mathbf T$ or $\mathbf{SL}$ and $\mathbf N$ satisfies the identity
\begin{equation}
\label{xxy=0}
x^2y=0\ldotp  
\end{equation} \qed
\end{proposition}

The next natural step in this investigation is the following description of a modular semigroup variety satisfying a permutational identity of length $3$.

\medskip

\begin{proposition}[\!\!{\mdseries\cite[Theorem~1.1]{Skokov-Vernikov-19}}]
\label{permut-3}
A semigroup variety $\mathbf V$ satisfying a permutational identity of length $3$ is a modular element in the lattice $\mathbb{SEM}$ if and only if $\mathbf V=\mathbf{M\vee N}$ where $\mathbf M$ is one of the  varieties $\mathbf T$ or $\mathbf{SL}$, while the variety $\mathbf N$ satisfies one of the following identity systems:

\begin{align*}
&xyz\approx zyx,\,x^2y\approx 0;\\
&xyz\approx yzx,\,x^2y\approx 0;\\
&xyz\approx yxz,\,xyzt\approx xzty,\,xy^2\approx 0;\\
&xyz\approx xzy,\,xyzt\approx yzxt,\,x^2y\approx 0. \qed
\end{align*} 
\end{proposition}

\medskip

Theorems~\ref{necessity} and~\ref{sufficiency} establish a necessary and a sufficient condition for a variety to be a modular element of the lattices $\mathbb{SEM}$ and $\mathbb{EPI}$. These conditions form the foundation for a complete description, which will be presented in the second part of our work. 

\section{Proof of Theorems: reduction to G-sets}
\label{reduction to G-sets}

We start with some definitions. A set $M$ is referred to as a \emph{check set} if the following properties hold:
\begin{enumerate}
\item[1.] $\alphabet(\mathbf u) = \alphabet(\mathbf v)$ for all $\mathbf u, \mathbf v \in M$;
\item[2.] $\mathbf u \not< \mathbf v$ for all $\mathbf u, \mathbf v \in M$;
\item[3.] if $\mathbf u \in M, \mathbf u \sim \mathbf v, \alphabet(\mathbf u) = \alphabet (\mathbf v)$ then $\mathbf v \in M$.
\end{enumerate}
A set $M$ is a \emph{weaker check set} if it satisfies the conditions 1. and 3.

\bigskip

Now, let us define $\I(\mathbf V)$ as the set of all elements $\mathbf u \in F$ such that the identity $\mathbf u \approx 0$ holds in the variety $\mathbf V$:

$$\I(\mathbf V) = \{\mathbf u \in F \,\,\,|\,\,\, \mathbf u \approx 0 \,\,\,\textnormal{holds in} \,\,\,\mathbf V\}.$$

\bigskip

Next, let $A$ be a non-empty set, $G$ be a group, and $\phi$ be a homomorphism from $G$ into
the full transformation group of $A$. Then $A$ may be considered as a unary
algebra with the set $G$ of operations where an operation $g \in G$ is defined
by the rule $g(x) = (\phi(g))(x)$ for every $x \in A$. Such algebras are known as
\emph{G-sets}.

The congruence lattice of a $G$-set $A$ is denoted by $\Con(A)$. Suppose $M$ is a weaker check set and $A=\alphabet(\mathbf u)$ for all $\mathbf u\in M$. For a permutation $\sigma\in S(A)$, we assume $\sigma$ acts on $M$ letterwise. It follows from the defitition of a weaker check set that $\sigma(\mathbf u)\in M$ for any $\sigma\in S(A)$ and $\mathbf u\in M$. Therefore, $M$ can be considered a $G$-set with $G=S(M)$. We refer to this $G$-set structure when we mention the congruence lattice $\Con(M)$.

\bigskip

Let $\sim_{\mathbf V}$ denote the fully invariant congruence on the free semigroup $F$ corresponding to a variety $\mathbf V$. The set of all fully invariant congruences on $F$ is a lattice, denoted by $\mathbb{FIC}$. It is a well-known fact that the lattice $\mathbb{SEM}$ is anti-isomorphic to the lattice $\mathbb{FIC}$.

Now we are going to prove a proposition that plays a crucial role in the proof of the main theorems.

\bigskip

\begin{proposition} 
\label{nil-nec}If a nil-variety $\mathbf V$ is a modular element in the lattice $\mathbb{SEM}$ or $\mathbb{EPI}$ then it satisfies the following conditions:
  \begin{enumerate}
  \item[i.] if $\mathbf V$ satisfies the identity $\mathbf u \approx \mathbf v$ but does not satisfy $\mathbf u \approx 0$ then $\mathbf u \approx \mathbf v$ is a substitutive identity;
  \item[ii.] for any check set $M$ such that $M \cap \I(\mathbf V) =\varnothing $, $ \sim_V \mid _M$~is a modular element of the lattice $\Con(M)$.
  \end{enumerate}
\end{proposition}
\begin{proof} The item i. holds by Proposition~\ref{SEM cmod nil-nec}. It remains to prove ii.

Let $M$ be some check set. Consider the collection $L_M$ of all nil-varieties $\mathbf{X}$ such that $M$ is a union of $\sim_{\mathbf{X}}$-classes. Notice that $L_M$ as a set of semigroup varieties is a sublattice of $\mathbb {SEM}$. Indeed, a semigroup variety $\mathbf X$ is in $L_M$ if and only if $\sim_{\mathbf X}\subseteq\alpha$ where $\alpha$ is the partition of $F$ into two classes $\{M, F\setminus M\}$. So we have $\sim_{\mathbf U}, \sim_{\mathbf W} \subseteq \alpha$ for any $\mathbf U, \mathbf W \in L_M$. Evidently, $\sim_{\mathbf U} \vee \sim_{\mathbf W} \subseteq \alpha$ and $\sim_{\mathbf U} \wedge \sim_{\mathbf W} \subseteq \alpha$. This means that $\mathbf U\vee \mathbf W \in L_M$ and $\mathbf U \wedge \mathbf W \in L_M$.

It follows from the item i. that $\mathbf V\in L_M$. It is clear that $\mathbf V$ is a modular element in this sublattice. Using Lemma~\ref{surjective homomorphism} and the fact that the condition of being a modular element is dual to itself, it is sufficient to show that the rule $\mathbf X \mapsto \sim_{\mathbf X} |_M$ defines a surjective anti-homomorphism from $L_M\rightarrow\Con(M)$ such that $\mathbf V\mapsto\sim_{\mathbf V} |_M$. 

Let $\phi$ be the anti-isomorphism between the lattices $\mathbb{SEM}$ and $\mathbb{FIC}$ restricted to the sublattice $L_M$ of $\mathbb{SEM}$. We denote by $\psi\colon \phi(L_M)\to \Con(M)$ the mapping that sends a relation $\sim$ to its restriction $\sim|_M$.
We will show that $\chi=\psi\cdot\phi$ is a surjective anti-homomorphism of $L_M$ onto $\Con(M)$. It is obvious that $\phi$ is an anti-isomorphism between $L_M$ and $\phi(L_M)$, so it remains to prove that $\psi$ is a surjective homomorphism. It is well known (see, for instance,~\cite{Gratzer-11}) that for the above mentioned partition $\alpha$ we have $$(\alpha ] \cong \Part(M) \times \Part (F \setminus M)$$ where $(\alpha]$ is the principal ideal of $\alpha$ in the partition lattice of $M$. This isomorphism maps any congruence $\beta \subseteq \alpha$ to the pair $(\beta | _M, \beta | _{F \setminus M})$. Hence $\psi$ is a homomorphism, acting as the projection to the first factor of the direct product. 

Now we are going to prove that $\psi$ is surjective. Take any $\beta \in \Con(M)$. Consider the variety $\mathbf Y$ defined by all identities $\mathbf u\approx\mathbf v$ where $(\mathbf u,\mathbf v)\in\beta$ and all identities $\mathbf u\approx 0$ where $\mathbf u>\mathbf v$ for some $\mathbf v\in M$. We will show that  $\mathbf Y \in L_M$ and $\sim_{\mathbf Y} | _M = \beta$. Let $\mathbf u$ be any word from $M$. It is sufficient to prove that, for any word $\mathbf v$ such that the identity $\mathbf u \approx \mathbf v$ holds in $\mathbf Y$, we have $(\mathbf u, \mathbf v) \in \beta$. Let the sequence of words
$$\mathbf u = \mathbf w_0, \mathbf w_1, \mathbf w_2, \dots, \mathbf w_n = \mathbf v$$ 
be a deduction of the identity $\mathbf u \approx \mathbf v$ from the defining identities of $\mathbf Y$. We prove that $\mathbf w_i \beta \mathbf u$ by induction on $i$.

\begin{description}
\item[Base case]{The statement is evident for $i = 0$. Indeed, if $i = 0$ then $\mathbf u = \mathbf w_i$}
\item[Inductive step]{Now let $n > 0$ and $\mathbf u = \mathbf w_0 \beta \mathbf w_1 \beta \mathbf w_2 \beta \dots \beta \mathbf w_i$. Let us prove that $\mathbf w_i \beta \mathbf w_{i+1}$. The identity $\mathbf w_i \approx \mathbf w_{i+1}$ directly follows from some $(\mathbf s, \mathbf t) \in \beta$. This means that, without loss of generality, $\mathbf w_i = \mathbf a \xi (\mathbf s) \mathbf b$, $\mathbf a \xi(\mathbf t) \mathbf b = \mathbf w_{i+1}$ for some $\xi\in\End(F)$. According to the definition, $\mathbf s \le \mathbf a \xi (\mathbf s) \mathbf b = \mathbf w_i\in M$. Recall that $\mathbf s \in M$. This implies that $\mathbf s \sim \mathbf a \xi(\mathbf s) \mathbf b$. Hence the words $\mathbf a$ and $\mathbf b$ are empty and $\xi$ is a permutation on $\alphabet(\mathbf s)$. In other words, $\xi |_{\alphabet(\mathbf s)} \in S(\alphabet(\mathbf s))$. Since $\beta$ is a congruence on $M$, this implies that $(\mathbf w_i, \mathbf w_{i+1})=(\xi(\mathbf s), \xi(\mathbf t)) \in \beta$, whence $(\mathbf u, \mathbf w_{i+1}) \in \beta$.

So, the map $\chi\colon L_M \rightarrow \Con (M)$ is a surjective anti-homomorphism and $\sim_{\mathbf V | _M} = \chi (\mathbf V)$. The variety $\mathbf V$, regardless of being a variety of semigroups or epigroups, is a modular element of the lattice of nilvarieties, and therefore of the lattice $L_M$.  By Lemma~\ref{surjective homomorphism}, $\sim_\mathbf V | _M$ is modular in $\Con(M)$.}
\end{description}

\end{proof}

\begin{proposition}
\label{nil-suf} Suppose a nil-variety $\mathbf V$ satisfies the condition i. of Proposition~\ref{nil-nec} and the condition 
\begin{enumerate}
\item[ii'.] for any weaker check set $M$ such that $M \cap \I(\mathbf V) =\varnothing $, $ \sim_V \mid _M$~is a modular element of the lattice $\Con(M)$.
\end{enumerate}
Then $\mathbf V$ is a modular element of both lattices $\mathbb{SEM}$ and $\mathbb{EPI}$.
\end{proposition}

\begin{proof} We are going to verify that $\mathbf V$ is a modular element in the
lattice $\mathbb {SEM}$. Let us fix varieties $\mathbf U, \mathbf W \in \mathbf {SEM}$ with $\mathbf U \subseteq \mathbf W$.
We will show that $$(\mathbf V \vee \mathbf U) \wedge \mathbf W = (\mathbf V \wedge \mathbf W) \vee \mathbf U.$$ It is clear
that $(\mathbf V \wedge \mathbf W) \vee \mathbf U \subseteq (\mathbf V \vee \mathbf U) \wedge \mathbf W$. It remains to verify that $(\mathbf V \vee \mathbf U) \wedge \mathbf W \subseteq (\mathbf V \wedge \mathbf W) \vee \mathbf U$. Let $(\mathbf V \wedge \mathbf W) \vee \mathbf U$ satisfy an identity $\mathbf u \approx \mathbf v$. We have to prove that this identity holds in $(\mathbf V \vee \mathbf U) \wedge \mathbf W$.

It is clear that the identity $\mathbf u \approx \mathbf v$ holds in $\mathbf U$ and $\mathbf V \wedge \mathbf W$. This means that there is a deduction of this identity $$\mathbf u = \mathbf w_0 \approx \mathbf w_1 \approx \dots \approx \mathbf w_k \approx \mathbf w_{k+1}= \mathbf v,$$ where $\mathbf w_i \approx \mathbf w_{i+1}$ holds either in $\mathbf V$ or in $\mathbf W$.

\textbf{Case 1}: $\mathbf u, \mathbf w_1, \mathbf w_2, \dots, \mathbf w_k, \mathbf v \not\in I(\mathbf V \wedge \mathbf W)$. Since $\mathbf V\wedge\mathbf W$ is a nilvariety, this implies $$\alphabet (\mathbf u) =\alphabet (\mathbf w_1)= \dots =\alphabet (\mathbf v).$$ Let us denote the set $\alphabet(\mathbf u)$ by $X$. Put $$M = \{\sigma (\mathbf u), \sigma (\mathbf w_1), \dots, \sigma(\mathbf v) | \sigma \in S(X)\}.$$ It is easy to see that $M$ is a weaker check set. Let $\alpha = \sim_{\mathbf V} \mid _M$, $\beta = \sim_{\mathbf U} \mid _M$, $\gamma = \sim_{\mathbf W} \mid _M$. Obviously $\alpha, \beta, \gamma \in \Con(M)$ and $\gamma \subseteq \beta$. This means that

\begin{align*}
& (\alpha \vee \gamma) \wedge \beta & = \,\,\, & (\alpha \wedge \beta) \vee \gamma & \alpha \,\, \text{is modular}\\
& &  = \,\,\, &  (\sim_V \mid _M \wedge \sim_U \mid _M) \vee \sim_W \mid _M & \\
& & \subseteq \,\,\, & (\sim_V  \wedge \sim_U) \vee \sim_W & \\
& & = \,\,\, &\sim_{(V \vee U) \wedge W} & \\
\end{align*}

Since, $(\mathbf u, \mathbf v) \in (\alpha \vee \gamma) \wedge \beta$ we have that $(\mathbf u, \mathbf v) \in \sim_{(V \vee U) \wedge W}$.

\textbf{Case 2}: there is a positive integer $p$ that $\mathbf u, \mathbf w_1, \mathbf w_2, \dots, \mathbf w_{p-1} \not\in I(\mathbf V \wedge \mathbf W)$ and $\mathbf w_{p}\in I(\mathbf V \wedge \mathbf W)$. This means that $\mathbf V \wedge \mathbf W$ satisfies the identity $\mathbf w_p \approx 0$.  

Let us prove that $\mathbf W$ satisfies the identity $\mathbf u \approx \xi_i (\mathbf w_i)$ for any natural $0 \le i \le p$ and some $\xi_i$ by induction on $i$. Let the identity $\mathbf w_{i-1} \approx \mathbf w_i$ hold in $\mathbf V$. It is clear that $\mathbf V$ does not satisfy the identity $\mathbf w_{i-1} \approx 0$. This means there exists an automorphism $\zeta$ such that $\mathbf w_{i-1} = \zeta(\mathbf w_i)$. Using the induction hypothesis, $\mathbf u \approx \xi_{i-1}(\mathbf w_{i-1}) = \xi_{i-1}\,\cdot \,\zeta (\mathbf w_{i})$. Let us denote $\xi_{p} = \xi_{p-1} \cdot \zeta$  in $\mathbf W$ and $\mathbf u \approx \xi_p (\mathbf w_p)$  in $\mathbf W$. 

Now let the identity $\mathbf w_{i-1} \approx \mathbf w_i$ hold in $\mathbf W$. Using the induction hypothesis, $\mathbf u \approx \xi_{i-1}(\mathbf w_{i-1})$ in $\mathbf W$. This means that $\mathbf W$ satisfies the identities $\xi_{i-1} ^{-1}(\mathbf u) \approx \mathbf w_{i-1} \approx \mathbf w_i$. So, $\mathbf u \approx \xi_{i-1}(\mathbf w_i)$ in $\mathbf W$. 

Denote $\mathbf s = \xi_{p}(\mathbf w_p)$. It is clear that the identity $\mathbf u \approx \mathbf s$ holds in $\mathbf W$ and the identity $\mathbf s \approx 0$ in $\mathbf V$. Analogously, there exists a word $\mathbf t$ such that the identity $\mathbf v \approx \mathbf t$ holds in $\mathbf W$ and the identity $\mathbf t \approx 0$ in $\mathbf V$. It is clear that $\mathbf V$ satisfies the identities $\mathbf s \approx 0 \approx \mathbf t$. Since $\mathbf U \subseteq \mathbf W$, the identities $\mathbf s \approx  \mathbf u$ and $\mathbf v \approx \mathbf t$ hold in $\mathbf U$. Recall that the identity $\mathbf u \approx \mathbf v$ holds in $(\mathbf V \wedge \mathbf W) \vee \mathbf U$. This means that this identity holds in $\mathbf U$. This implies that the identity $\mathbf s \approx \mathbf t$ holds in $\mathbf U \vee \mathbf V$. So, 
\begin{align*}
&\mathbf u \approx \mathbf s & \text{holds in} \,\,\,& \mathbf W,\\
&\mathbf s \approx \mathbf t & \text{holds in}\,\,\, & \mathbf U \vee \mathbf V,\\
& \mathbf t \approx \mathbf v & \text{holds in}\,\,\, & \mathbf W.
\end{align*}

One can conclude that $\mathbf u \approx \mathbf v$ holds in $\mathbf W \wedge (\mathbf U \vee \mathbf V)$.

\end{proof}

\section{Proof of Theorems: G-sets}
\label{G-sets}

A $G$-set $A$ is said
to be \emph{transitive} if, for all $a, b \in A$, there exists $g \in G$ such that $g(a) = b$. A transitive $G$-subset of a $G$-set $A$ is called an \emph{orbit} of $A$. Clearly, any $G$-set is a disjoint union of its orbits. Let $\Orb(A) $ be the set of all orbits of $A$. For simplicity of notation, we will assume the set $\Orb(A) $ to be finite throughout this section. All results can be generalized to the $G$-set with infinitely many orbits without any differences.

Let $\alpha \in \Con(A)$ and $B$ and $C$ be distinct orbits in $A$. We say that $\alpha$ \emph{isolates}
$B$ if $x \in B$ and $x \alpha y$ imply $y \in B$. We say that $\alpha$ \emph{connects} 
orbits $B$ and $C$ of $A$ if there are elements $b \in B$ and $c \in C$ with
$b \alpha c$.  A congruence $\alpha$ is said to be \emph{simple} if it isolates every orbit.  Let $\SCon(A)$ denote the set of all simple congruences of $A$. 

\medskip

\medskip

Let $\omega$ be the equivalence relation of being in the same orbit. It is clear that $\omega$ is a congruence. It is easy to see that the principal ideal of this congruence in the lattice $\Con(A)$ is the set of all simple congruences. We will show that this principal ideal is the direct product of the congruence lattices of the orbits of $A$.

\medskip

\begin{lemma}
\label{SCon is a direct product}
Let $A$ be a $G$-set and let $Orb(A) = \{A_i \mid 1\le i\le n\}$. The rule 
\begin{equation}
\label{Scon isomorphism}
f(\alpha)= (\alpha_1, \alpha_2, \dots, \alpha_n)
\end{equation}
where $\alpha_i$ is the restriction of a congruence $\alpha \in \SCon(A)$ to the orbit $A_i$, defines an isomorphism $f$ between the lattice $\SCon(A)$ and the direct product of lattices $\Con(A_i)$. 
\end{lemma}

\begin{proof}
In the lattice $\Part(A)$, the ideal $(\omega]_{\Part(A)}$ is isomorphic to the direct product of the partition lattices of $\omega$-classes, which are the orbits of $A$. The structure of this isomorphism is given by~(\ref{Scon isomorphism}). Since the lattice $(\omega]_{\Con(A)}$ is a sublattice of $(\omega]_{\Part(A)}$, it is isomorphic to some sublattice of this direct product. It remains to show that it is precisely the direct product of congruence lattices of the orbits. Indeed, it is straightforward to verify that the restriction of a congruence to an orbit is itself a congruence. Furthermore, any set of congruences on the orbits $\{\alpha_i | 1\le i\le n\}$ is the image of $\cup \alpha_i$, which is a congruence on $A$. This implies that $f$ is an isomorhpism. 
\end{proof}

It is well known (see, for example, ~\cite[Lemma 4.20]{McKenzie-McNulty-Taylor-87}) that if $A$
is a transitive G-set, then the lattice $\Con(A)$ is isomorphic to an interval of
the subgroup lattice of G (specifically, $\Con(A) \simeq [\Stab_G(a), G]$, where $a$
is an arbitrary element in $A$, and $\Stab_G(a) =\{g \in G \,\,|\,\, g(a) = a\}$). Obviously, if $\Stab_G(a) = 1$ for every $a$ and $A_i$ is an orbit of $A$ then $\Con(A_i) \simeq \Sub(G)$.  This leads to the following lemma.

\medskip

\begin{lemma} 
\label{SCon-Sub} If $\Stab(x)=1$ for all $x\in A$ then
$$\SCon (A) = \Sub(G) \times  \Sub(G) \times \dots \times \Sub(G) = (\Sub(G)) ^ n. $$ \qed
\end{lemma}

\medskip

Throughout this section, we will always assume the condition~\ref{SCon-Sub} holds.

\medskip

\begin{lemma}
\label{stab in transitive} Let $A$ be a transitive G-set, and suppose $\Stab_G(a)=1$ for every $a \in A$. For any $x, y \in A$, there exists a unique element $g \in G$ such that $x = g(y)$. 
\end{lemma}

\begin{proof}
Assume $g_1(x) = y$ and $g_2(x) = y$ for some $g_1, g_2 \in G$. This implies that  $x = g_2^{-1} (y)$ and $g_1(g_2^{-1}(y)) = y$. Since  $\Stab_G(x)=\Stab_G(y)=1$, we conclude that $g_1 g_2^{-1} = 1$, which means $g_1 = g_2$.
\end{proof}

\medskip

Let $\mathcal T$ be a transversal of $Orb(A)$ (a set containing exactly one element from each orbit of $A$). For any element $x \in A$, let $G(x)$ denote the orbit containing $x$. If the transversal $\mathcal T$ is the set ${x_1, x_2, \dots, x_n}$, then $Orb(A) = {G(x_1), G(x_2), \dots, G(x_n)}$. We say that a congruence $\alpha$ is \emph{coordinated with the transversal} if $x_i \, \alpha \, x_j$ for any points from the transversal such that $\alpha$ connects the orbits $G(x_i)$ and $G(x_j)$.  Let $\Con_{\mathcal T}(A)$ denote the set of all congruences of $A$ coordinated with the transversal $\mathcal T$.

\medskip

\begin{lemma}
\label{Con_T is a lattice} For any G-set A and any transversal $\mathcal T$, the set $\Con_{\mathcal T}(A)$ is a sublattice of $\Con(A)$.
\end{lemma}
\begin{proof}
Let $\alpha, \beta \in \Con_{\mathcal T}(A)$ and let $\gamma = \alpha \wedge \beta$. Suppose that $\gamma$ connects the orbits $B$ and $C$. Clearly, both $\alpha$ and $\beta$ connect the orbits $B$ and $C$. Let $b \in B$ and $c \in C$ are elements from the transversal $\mathcal T$. It follows that $b \alpha c$ and $b \beta c$, which implies that $b \gamma c$.

Now let $\delta = \alpha \vee \beta$ and suppose that $\delta$ connects the orbits $B$ and $C$. This means there exist elements $b \in B$ and $c \in C$ such that $b \delta c$. Consider elements $e_1, e_2, \dots e_n \in A$ such that $$b=e_1 \,\alpha \, e_2 \, \beta \, e_3 \, \alpha \, \dots \beta \, e_n = c. $$ Since the congruences $\alpha,\beta\in\Con_{\mathcal T}(A)$, we can replace each element $e_i$ with the element $e_i'$ from $\mathcal T \cap G(e_i)$. This gives $$e_1' \,\alpha \, e_2' \, \beta \, e_3' \, \alpha \, \dots \beta \, e_n'. $$ Thus, $\delta$ is coordinated with the transversal $\mathcal T$.
\end{proof}

For any congruence $\alpha$ on $A$, we define a binary relation $\alpha^{*}$ on $Orb(A)$ by the following rule: for $B, C \in Orb(A)$, we write $B \,\alpha^{*}\, C$ if and only if either $B = C$ or $\alpha$ connects $B$ and $C$. 

It is easy to verify that $\alpha^{*}$ is an equivalence relation on $Orb(A)$. Reflexivity and symmetry are obvious. To show transitivity, suppose that $B \,\alpha^{*}\, C \,\alpha^{*}\, D$ for some orbits $B$, $C$, and $D$. We need to prove that $B \,\alpha^{*}\, D$.

By definition, there exist elements $b \in B$, $d \in D$, and $c_1, c_2 \in C$ such that $b \,\alpha\, c_1$ and $c_2 \,\alpha\, d$. Since $c_1$ and $c_2$ lie in the same orbit, there exists $g \in G$ such that $c_1 = g(c_2)$. As $\alpha$ is a congruence, we also have $g(c_2) \,\alpha\, g(d)$.

Therefore,
\[
b \,\alpha\, c_1 = g(c_2) \,\alpha\, g(d),
\]
whence $b \,\alpha\, g(d)$, so $B \,\alpha^{*}\, D$.

Note that, for any transversal $\mathcal T$ and a congruence $\sigma$ coordinated with $\mathcal T$, the equivalence relation $\sigma^{*}$ defines an equivalence relation on $\mathcal T$. We denote this equivalence relation by the same symbol, $\sigma^{*}$.

If $A$ is a G-set and $\alpha \in \Con(A)$ then for each $x \in A$ we define
$$\Stab_{\alpha} (x) = \{g \in G \,\,\, | \,\,\, x \, \alpha \,g(x)\}.$$
It is clear that $\Stab_{\alpha} (x)$ is a subgroup of $G$, and it is called the $\alpha$-\emph{stabilizer} of the element
$x$ in $A$.

Suppose we have a fixed $G$-set $A$ with $n$ orbits $A_1,a_2,\dots,A_n$. A \emph{code} is a sequence $(\sim,H_1,\dots,H_n)$ where $\sim$ is an equivalence relation on $\Orb(A)$ and $H_1,\dots,H_n$ are subgroups of $G$. A code is \emph{proper} if $H_i=H_j$ whenever $A_i\sim A_j$. For a transversal $\mathcal T  = (x_1, x_2, \dots , x_n)$ and a congruence $\sigma\in\Con_{\mathcal T}(A)$, we define $\mathcal T$-\emph{code} of $\sigma$ as the code $$\Code_{\mathcal T}(\sigma)=(\sigma^* \,|\,\,\Stab_{\sigma} (x_1), \, \Stab_{\sigma} (x_2), \, \dots ,\, \Stab_{\sigma} (x_n) ).$$ The set of all proper codes will be denoted as $\mathbb{PC}$.

\medskip

\begin{lemma}
\label{coding congruence} The rule $\sigma\mapsto\Code_{\mathcal T}(\sigma)$ defines a surjective mapping $$\Code_{\mathcal T}\colon\Con_{\mathcal T}\rightarrow\mathbb{PC}$$

\end{lemma}
\begin{proof} First, we prove that $\Code_{\mathcal T}(\sigma)\in\mathbb{PC}$ for any $\sigma\in\Con_{\mathcal T}$. Let $x$ and $y$ be points from the transversal $\mathcal{T}$, and let $X$ and $Y$ be orbits such that $x \in X$ and $y \in Y$. We aim to show that if $X \,\sigma^{*}\, Y$, then $\Stab_{\sigma}(x) = \Stab_{\sigma}(y)$.

By symmetry, it suffices to prove that $\Stab_{\sigma}(x) \subseteq \Stab_{\sigma}(y)$. Consider any element $g \in \Stab_{\sigma}(x)$. By definition, this means $g(x) \,\sigma\, x$.

Since $\sigma$ is coordinated with the transversal $\mathcal{T}$ and $\sigma^{*}$ connects the orbits $X$ and $Y$, it follows that $x \,\sigma\, y$. Moreover, as $\sigma$ is a congruence, we have $g(x) \,\sigma\, g(y)$.

Hence,
\[
y \,\sigma\, x \,\sigma\, g(x) \,\sigma\, g(y),
\]
which implies $y \,\sigma\, g(y)$, and therefore $g \in \Stab_{\sigma}(y)$. 

Now we will prove that $\Code_{\mathcal T}$ is surjective. We will construct a congruence $\sigma$ corresponding to a given code
\[
(\pi \mid H_1,\, H_2,\, \dots,\, H_n)\in\mathbb{PC},
\]
such that the code
\[
(\pi \mid H_1,\, H_2,\, \dots,\, H_n)
\]
coincides with the code
\[
(\sigma^{*} \mid \Stab_{\sigma}(x_1),\, \Stab_{\sigma}(x_2),\, \dots,\, \Stab_{\sigma}(x_n)).
\]

Consider all pairs of points $(x_i,x_j)$ from the transversal $\mathcal{T} = (x_1, x_2, \dots, x_n)$ such that $X_i \,\pi\, X_j$, where $X_i$ and $X_j$ are the orbits containing $x_i$ and $x_j$, respectively. Additionally, for each $x_i \in \mathcal{T}$ and each $g \in H_i$, consider the pair $(x_i, g(x_i))$. The congruence $\sigma$ is defined as the congruence generated by all such pairs.

First, we show that $\pi = \sigma^{*}$. Consider the equivalence relation $\delta$ on $A$ defined as follows: $x \,\, \delta \,\, y$ if and only if $G(x) \,\, \pi \,\, G(y)$. It is clear that $\delta$ is a congruence. Since all pairs from the generating set of $\sigma$ are in $\delta$, it follows that $\sigma \subseteq \delta$. Moreover, $\delta$ does not connect orbits not related by $\pi$. Consequently,  $\sigma$ does not connect orbits not related by $\pi$, and  thus $\sigma^{*} = \pi$. 

Now we show that for each point $x$ from the transversal $\mathcal T$, we have that $\Stab_{\sigma}(x) = H_x$. By construction of $\sigma$, $H_x \subseteq\Stab_{\sigma}(x)$. We need to prove the reverse inclusion. By definition, if $g \in \Stab_{\sigma}(x)$, then $x \,\, \sigma \,\, g(x)$. There exists a number of points $y_1, y_2, \dots, y_k \in A$ such that $y_1 = x$, $y_k = g(x)$ and each pair $(y_i, y_{i+1})$ is obtained by the action of some element $h$ on a pair from the generating set of pairs. Let $z_1, z_2, \dots, z_k \in \mathcal T$ where $z_i$ and $y_i$ lie in the same orbit. For each $y_i$ there exists $g_i$ such that $y_i = g_i (z_i)$. Now we show that $g_i \in H_x$ for all $i$. We proceed by induction on $i$. 
\begin{description}
\item[Base case]{The statement is evident for $i = 1$.}
\item[Inductive step]{Suppose $g_i \in H_x$.  We need to show that $g_{i+1} \in H_x$. Consider the pair of elements $y_i$ and $y_{i+1}$. It is known that $y_i = h(s)$ and $y_{i+1} = h (t)$, where $(s, t)$ is one of the pairs from generating set.
\begin{description}
\item[Case 1] {If $s,t \in \mathcal T$, then $s = z_i$ and $t = z_{i+1}$. By assumption, $y_{i+1} = h(z_{i+1})$. But on the other hand $y_{i+1} = g_{i+1}(z_{i+1})$. Hence, $z_{i+1} =h^{-1}\cdot g_{i+1}(z_{i+1})$. Since $\Stab (z_{i+1}) =1$, it follows that $g_{i+1} = h \in H_x$.}

\item[Case 2] {If $s = g(t)$, we may assume without loss of generality that $(y_i, y_{i+1}) = (h (z_i), h (g(z_i)))$. Since $y_i$ and $y_{i+1}$ are in the same orbit, we have $z_i = z_{i+1}$ and $g_i = h$ and $g_{i+1} = h \cdot g$. By the induction hypothesis, $h \in H_x$, and by definition, $g \in H_x$. Therefore, $g_{i+1} \in H_x$. }
\end{description}}
\end{description}

An analogous argument applies in the case where $(y_i, y_{i+1}) = (h(g(z_i)), h(z_i))$.

Thus, the proof is complete. 
\end{proof}

The set $\mathbb{PC}$ can be considered as a partially ordered set. Namely, we put

$$ (\pi_1\,|\,\,H_{1}, \, H_{2}, \, \dots ,\, H_{n}) \le (\pi_2\,|\,\,P_{1}, \, P_{2}, \, \dots ,\, P_{n} )$$ if and only if $\pi_1 \subseteq \pi_2$ and $H_{i} \le P_{i}$ for every $1 \le i \le n$.

\medskip

\begin{lemma}
\label{order on codes} The mapping $\Code_{\mathcal T}$ is an isomorphism of partially ordered sets $\Con_{\mathcal T}$ and $\mathbb{CR}$.
\end{lemma}

\begin{proof} We suppose $\mathcal T=\{x_1,\dots,x_n\}$. The mapping $\Code_{\mathcal T}$ is surjective by Lemma~\ref{coding congruence}. It remains to prove that, for any congruences $\sigma_1$ and $\sigma_2$ coordinated with $\mathcal T$, $\sigma_1 \subseteq \sigma_2$ if and only if $\Code_{\mathcal T}(\sigma_1)\le\Code_{\mathcal T}(\sigma_2)$. In particulat, this will imply that $\Code_{\mathcal T}$ is injective. We put $H_i=\Stab_{\sigma_1}(x_i)$ and $P_i=\Stab_{\sigma_2}(x_i)$.

\emph{Necessity}. Suppose $\sigma_1 \subseteq \sigma_2$. It is clear that $\sigma_1^{*} \subseteq \sigma_2^{*}$. Now, let $g \in H_i$ for some $1 \le i \le n$. We know that $x_i \, \sigma_1 \, g(x_i)$. Since $\sigma_1 \subseteq \sigma_2$, we also have $x_i \, \sigma_2 \, g(x_i)$, which implies $g \in P_i$. Thus, $H_i \le P_i$.

\emph{Sufficiency}. Suppose $\sigma_1^{*} \subseteq \sigma_2^{*}$ and $H_i \le P_i$ for every $1 \le i \le n$. We will show that $y_1 \, \sigma_1 \,y_2$ implies $y_1 \, \sigma_2 \, y_1$. Without loss of generality, assume that $x_1 \in G(y_1)$ and $x_2 \in G(y_2)$. This means that there exist $g_1, g_2 \in G$ such that $g_1(x_1) = y_1$ and $g_2(x_2) = y_2$. Since $\sigma_1$ is coordinated with $\mathcal T$ and $y_1 \sigma_1 y_2$, we have $x_1 \sigma_1 x_2$. Since $\sigma_1$ is a congruence, we also have $g_1(x_1) \sigma_1 g_1 (x_2)$. Hence, 

$$g_2(x_2) = y_2 \,\, \sigma_1 \,\, y_1 = g_1(x_1) \,\, \sigma_1 \,\, g_1 (x_2).$$

We have $g_1^{-1} \cdot g_2 (x_2) \,\, \sigma_1 \,\, x_2$. This implies $g_1^{-1} \cdot g_2\in H_2$. But we know that $H_2 \le P_2$. Hence, $g_1^{-1} \cdot g_2 (x_2) \,\, \sigma_2 \,\, x_2$. Thus, $g_1(x_2) \sigma_2 g_2(x_2)$. Since $\sigma_1^{*} \subseteq \sigma_2^{*}$, we have $x_1 \sigma_2 x_2$.  Since $\sigma_2$ is a congruence, it implies that $g_1(x_1) \sigma_2 g_1 (x_2)$. Therefore,

$$y_1 = g_1(x_1) \,\, \sigma_2 \,\, g_1(x_2) \,\, \sigma_2 \,\, g_2(x_2) = y_2.$$ 
\end{proof}

Lemma~\ref{order on codes} immediately implies

\medskip

\begin{corollary}
\label{PC is a lattice} The partially ordered set $\mathbb{PC}$ is a lattice. \qed
\end{corollary}

The following lemma directly describes lattice operations on $\mathbb{CR}$.

\medskip

\begin{lemma}
\label{meet and join of codes} Let $C_1=(\pi_1\,|\,\,H_{1}, \, H_{2}, \, \dots ,\, H_{n})$ and $C_2=(\pi_2\,|\,\,P_{1}, \, P_{2}, \, \dots ,\, P_{n})$ be two proper codes. The codes of $C_1 \wedge C_2$ and $C_1\vee C_2$ can be calculated by the rules
$$C_1\wedge C_2=(\pi_1\wedge \pi_2\,|\,\, H_{1} \wedge P_{1}, H_{2} \wedge P_{2} \, \dots ,\, H_{n} \wedge P_{n})$$ 
and
$$C_1\vee C_2=(\pi_1\vee \pi_2\,|\,\,K_{1}, \, K_{2}, \, \dots ,\, K_{n})$$
where
$$K_{i} = \left(\bigvee_{j\colon(A_i,A_j) \in \pi_1\vee\pi_2}{H_j}\right) \,\, \vee \,\, \left(\bigvee_{j\colon(A_i,A_j) \in \pi_1\vee\pi_2}{P_j}\right).$$
\end{lemma}

\begin{proof} We denote the $i$-th subgroup in $C_1\vee C_2$ by $K_i$ and the subgroup
$$\left(\bigvee_{j\colon(A_i,A_j) \in \pi_1\vee\pi_2}{H_j}\right) \,\, \vee \,\, \left(\bigvee_{j\colon(A_i,A_j) \in \pi_1\vee\pi_2}{P_j}\right)$$
by $L_i$. The equality $K_i=L_i$ is the only non-trivial statement of the lemma. Since the code $C_1\vee C_2$ is proper, $(A_i,A_j) \in \pi_1\vee\pi_2$ implies $K_i=K_j$. Since $C_1,C_2\le C_1\vee C_2$, we have $H_j,P_j\le K_j$. Hence, $L_j\le K_j$. Furthermore, it is easy to verify that $C=(\pi_1\vee\pi_2\,|\,\,L_{1}, \, L_{2}, \, \dots ,\, L_{n})$ is a proper code and $C_1,C_2\le C$. Hence $C_1\vee C_2\le C$ and $K_j\le L_j$.
\end{proof}

\medskip

\begin{lemma}
\label{stabilizers are conjugated} For any congruence $\sigma$, $\sigma$-stabilizers of points in the same orbit are conjugated.
\end{lemma}
\begin{proof}
Let $x$ and $y$ be different points from one orbit. Let $H_x$ and $H_y$ be their $\sigma$-stabilizers. Take an element $h\in H_x$. There exists $g \in G$ such that $x = g(y)$. It is clear that $g(y)\,\sigma\,hg(y)$. So, $y \,\, \sigma g^{-1}hg(y)$. This means that $g^{-1}hg \in H_y$. Hence $g^{-1}H_xg \le H_y$. Analogously, we can show that $gH_yg^{-1}\le H_x$, whence $H_y\le g^{-1}H_xg$.
\end{proof}

\section{Proof of Theorems}
\label{Proof of Theorems}

\begin{lemma}
\label{M3} Let $X$ be a $G$-set such that $g(x)\neq x$ for any $x\in X$ and any non-identity $g\in G$. Suppose the lattice $\Sub(G)$ contains a sublattice $\{I,H_1,H_2,H_3,J\}\sim M_3$ with zero $I$ and unit $J$. Suppose $\alpha\in\Con(X)$ is a simple congruence such that $\Stab_{\alpha}(x_1)=H_1$ and $\Stab_{\alpha}(x_2)=H_2$ for some points $x_1,x_2\in X$ from two different orbits. Then $\alpha$ is not a modular element of the lattice $\Con(X)$.
\end{lemma}
\begin{proof} Let $\Orb(X)=\{A_1,A_2,\dots,A_n\}$ and $x_1\in A_1$, $x_2\in A_2$. It is clear that
$$\Code_{\mathcal T}(\alpha)=(0\,|\,\,H_1,H_2,K_3,\dots,K_n)$$
for some subgroups $K_3,\dots,K_n\in\Sub(G)$. Let $\varepsilon$ be the partition of $\Orb(X)$ with $\{A_1,A_2\}$ being the only non-singleton class. Take two congruences $\beta,\gamma\in\Con_{\mathcal T}(G)$ such that
$$\Code_{\mathcal T}(\beta)=(\varepsilon\,|\,\,I,I,K_3,\dots,K_n)$$
and
$$\Code_{\mathcal T}(\gamma)=(\varepsilon\,|\,\,H_3,H_3,K_3,\dots,K_n).$$
It is easy to see that $\beta\subseteq\gamma$,
$$\Code_{\mathcal T}(\alpha\wedge\gamma)=(0\,|\,\,I,I,K_3,\dots,K_n)$$
and
$$\Code_{\mathcal T}(\alpha\vee\beta)=(\varepsilon\,|\,\,J,J,K_3,\dots,K_n).$$
Hence $\alpha\wedge\gamma\subseteq\beta$ and $\gamma\subseteq\alpha\vee\beta$. It means that the set 
$$\{\alpha,\beta,\gamma,\alpha\wedge\gamma,\alpha\vee\beta\}$$
is a sublattice of $\Con(X)$ isomorphic to $N_5$. Hence $\alpha$ is non-modular in $\Con(x)$.
\end{proof}

\begin{proposition}
\label{Stab in G-set is a modular} Let $X$ be a G-set such that $G = S_n$ and for any $x \in X$ and any nontrivial $g \in G$ we have $g(x) \neq x$. Let $\sigma$ be a simple congruence from $\Con(X)$. Then $\sigma$ is a modular element in the lattice $\Con(X)$ if and only if the following holds
\begin{itemize}
\item[\textup{i)}] $\Stab_{\sigma} (x)$ is a modular element in the lattice $\Sub(G)$ for any $x \in X$;
\item[\textup{ii)}]  there are no points $x$ and $y$ from different orbits such that the following holds
\begin{itemize}
\item[a)] $\Stab_{\sigma} (x), \Stab_{\sigma} (y) \in \{ T_{12},  T_{23},  T_{13}\}$;
\item[b)] $\Stab_{\sigma} (x) \in \{ T_{12},  T_{23},  T_{13}\}, \Stab_{\sigma} (y) =  A_3$;
\item[c)] $\Stab_{\sigma} (x), \Stab_{\sigma} (y) \in \{ I_{12,34},  I_{13,24},  I_{14,23}\}$;
\item[d)] $\Stab_{\sigma} (x) \in \{ I_{12,34},  I_{13,24},  I_{14,23}\}, \Stab_{\sigma} (y) =  A_4$.
\end{itemize}
\end{itemize} \qed
\end{proposition}

\begin{proof}\emph{Necessity}. The property~(i) follows immediately from Lemma~\ref{SCon-Sub} and Lemma~\ref{surjective homomorphism}.

We now proceed to prove~(ii). Arguing by contradiction, it is sufficient to check that the conditions of Lemma~\ref{M3} are satisfied in all cases a)--d), that is, to find a sublattice $L$ isomorphic to $M_3$ in $\Sub(G)$.

\medskip

\emph{Case} \textbf{a)} By Lemma~\ref{stabilizers are conjugated}, we can suppose that $\Stab_{\sigma}(x)=T_{12}$ and $\Stab_{\sigma}(y)=T_{13}$. We take the sublattice $\{T,T_{12},T_{13},T_{23},S_3\}$. 

\medskip

\emph{Case} \textbf{b)} By Lemma~\ref{stabilizers are conjugated}, we can suppose that $\Stab_{\sigma}(x)=T_{12}$. We take the sublattice $\{T,T_{12},T_{13},A_3,S_3\}$. 

\medskip

\emph{Case} \textbf{c)} By Lemma~\ref{stabilizers are conjugated}, we can suppose that $\Stab_{\sigma}(x)=I_{12,34}$ and \\$\Stab_{\sigma}(y)=T_{13,24}$. We take the sublattice $\{V_4,I_{12,34},I_{13,24},I_{14,23},S_4\}$. 

\medskip

\emph{Case} \textbf{d)} By Lemma~\ref{stabilizers are conjugated}, we can suppose that $\Stab_{\sigma}(x)=I_{12,34}$. We take the sublattice $\{V_4,I_{12,34},I_{13,24},A_4,S_4\}$. 

\bigskip

\emph{Sufficiency}. 
We are going to prove that for any congruences $\beta$ and $\gamma$ such that $\beta \subseteq \gamma$, the following equality holds:
\[
(\alpha \vee \beta) \wedge \gamma = (\alpha \wedge \gamma) \vee \beta.
\]

We will construct a transversal $\mathcal{T}$ which is coordinated with all of the congruences $\alpha$, $\beta$, and $\gamma$. 

Since $\alpha$ is a simple congruence, it is coordinated with any transversal. Thus, it suffices to construct a transversal that is coordinated with both $\beta$ and $\gamma$.

We begin with the following lemma.

\medskip

\begin{lemma}
\label{Gamma class}
Every $\gamma$-class contains at least one point from each orbit in the corresponding $\gamma^{*}$-class.
\end{lemma}

\begin{proof}
Let $x$ be a point in a $\gamma$-class $\Gamma$. Let $Y$ be an orbit in the corresponding $\gamma^{*}$-class, distinct from $G(x)$. Since $Y \,\gamma^{*}\, G(x)$, by the definition of $\gamma^{*}$, there exist elements $z \in G(x)$ and $y \in Y$ such that $z \,\gamma\, y$.

As $x$ and $z$ lie in the same orbit, there exists an element $g \in G$ such that $x = g(z)$. Since $\gamma$ is a congruence, $z \,\gamma\, y$ implies $g(z) \,\gamma\, g(y)$, and hence $x \,\gamma\, g(y)$. Therefore, the $\gamma$-class $\Gamma$ contains at least one point $g(y)$ from $Y$.
\end{proof}

Analogously, every $\beta$-class contains at least one point from each orbit of the corresponding $\beta^{*}$-class.

Now consider a $\gamma^{*}$-class $\Gamma_*$. Let $\Gamma$ be a $\gamma$-class contained in the union of orbits from $\Gamma_*$. Evidently, each $\gamma^{*}$-class is a union of several $\beta^{*}$-classes. For each $\beta^{*}$-class $B_*$ within $\Gamma_*$, fix a $\beta$-class $B$ such that $B \subseteq (\cap B_*) \cap \Gamma$. From each orbit that intersects $B$, select a single representative point. The set of all such chosen points forms the desired transversal~$\mathcal{T}$.

By Lemma~\ref{Con_T is a lattice}, both $(\alpha \vee \beta) \wedge \gamma$ and $(\alpha \wedge \gamma) \vee \beta$ are congruences coordinated with the transversal~$\mathcal{T}$. By definition,
\[
\left( ((\alpha \vee \beta) \wedge \gamma)_{\mathcal{T}}^{*} \,\middle|\, \Stab_{(\alpha \vee \beta) \wedge \gamma}(x_1), \, \Stab_{(\alpha \vee \beta) \wedge \gamma}(x_2), \, \dots, \, \Stab_{(\alpha \vee \beta) \wedge \gamma}(x_n) \right)
\]
is the code for the congruence $(\alpha \vee \beta) \wedge \gamma$, and
\[
\left( ((\alpha \wedge \gamma) \vee \beta)_{\mathcal{T}}^{*} \,\middle|\, \Stab_{(\alpha \wedge \gamma) \vee \beta}(x_1), \, \Stab_{(\alpha \wedge \gamma) \vee \beta}(x_2), \, \dots, \, \Stab_{(\alpha \wedge \gamma) \vee \beta}(x_n) \right)
\]
is the code for the congruence $(\alpha \wedge \gamma) \vee \beta$.

Now it is sufficient to show that these codes are equal. 

At first we need to show that ${((\alpha \vee \beta) \wedge \gamma)}^{*} = {((\alpha \wedge \gamma) \vee \beta)}^{*}$. 

\begin{align*}
{((\alpha \vee \beta) \wedge \gamma)}^{*} & =  {(\beta \wedge \gamma)}^{*} &  \,\, \text{since} \,\,\, \alpha \,\,\, \text{is simple}\\
 & =  \beta ^*&\,\, \text{because} \,\,\, \beta \subseteq \gamma 
\end{align*}

Since $\alpha$ is simple, ${(\alpha \wedge \gamma)}^{*}$ is trivial, and ${((\alpha \wedge \gamma) \vee \beta)}^{*} =  \beta ^*$.

Now, let us show that $\Stab_{((\alpha \vee \beta) \wedge \gamma)} (x_i) = \Stab_{((\alpha \wedge \gamma) \vee \beta)} (x_i)$ for all $x_i \in \mathcal T$. Note that, under the conditions i) and ii), all subgroups $\Stab_{\alpha}(x_i)$ are comparable. We have

\begin{align*}
\Stab_{(\alpha \wedge \gamma) \vee \beta} (x_i)
 & = \left(\bigvee_{j\colon(A_i,A_j) \in \beta} \Stab_{\alpha} (x_j) \wedge \Stab_{\gamma}(x_i)\right) \vee \Stab_{\beta} (x_i)\\[1.5ex]
 & =\left(\max_{j\colon(A_i,A_j) \in \beta} \Stab_{\alpha} (x_j) \wedge \Stab_{\gamma}(x_i)\right) \vee \Stab_{\beta} (x_i) \\[1.5ex]
  & = \left(\max_{j\colon(A_i,A_j) \in \beta} \Stab_{\alpha} (x_j) \vee \Stab_{\beta} (x_i)\right) \wedge \Stab_{\gamma}(x_i)  \\[1.5ex]
  & = \left(\bigvee_{j\colon(A_i,A_j) \in \beta}\Stab_{\alpha} (x_j) \vee \Stab_{\beta} (x_i)\right) \wedge \Stab_{\gamma}(x_i)  \\[1.5ex]
  & = \Stab_{((\alpha \vee \beta) \wedge \gamma)} (x_i).
\end{align*}
\end{proof}

Now let us finish the proofs of the main theorems.

\bigskip

\emph{Proof of the Theorem}~\ref{necessity}. Let $\mathbf V$ be a nil-variety which is a modular element of the lattice $\mathbb {SEM}$ or $\mathbb {EPI}$. 

Property $a)$ follows from Proposition~\ref{SEM cmod nil-nec}. To prove b), consider the set $X$ of all words $\sigma(\mathbf u)$ where $\sigma\in S(\alphabet(\mathbf u))$.  The set $X$ is a check set, so Propositions~\ref{nil-nec} and~\ref{Stab in G-set is a modular} apply.

For c), consider the set of all words $\sigma(\mathbf u)$ and $\sigma(\mathbf v)$  where $\sigma\in S(\alphabet(\mathbf u))$ and apply Propositions~\ref{nil-nec} and~\ref{Stab in G-set is a modular} again.\qed

\bigskip

\emph{Proof of the Theorem}~\ref{sufficiency}. This directly follows from Propositions~\ref{nil-suf} and~\ref{Stab in G-set is a modular}.\qed

\section*{Acknowledgements}

The authors thank M. V. Volkov and B. M. Vernikov for many useful discussions over the years on semigroup and epigroup varieties, as well as for their comments and valuable remarks on a preliminary version of this paper. The authors are also grateful to A. E. Guterman for his remarks on the manuscript and for his attention to the work.

\end{document}